\documentclass{article}

\usepackage[english]{babel}
\usepackage{balance} % for balancing columns on the final page
\usepackage[utf8]{inputenc}
\usepackage{amsmath}
\usepackage{amsthm}
\usepackage[mathcal]{euscript}
\usepackage[T1]{fontenc}
\usepackage{url}            % simple URL typesetting
\usepackage{booktabs}       % professional-quality tables
\usepackage{amsfonts}       % blackboard math symbols
\usepackage{nicefrac}       % compact symbols for 1/2, etc.
\usepackage{microtype}      % microtypography
\usepackage{tikz}
\usepackage[inline]{enumitem}
\usepackage{multirow}
\usepackage{diagbox}
\usepackage{mathtools}
\usepackage[letterpaper,top=2cm,bottom=2cm,left=3cm,right=3cm,marginparwidth=1.75cm]{geometry}

\usepackage{amsmath}
\usepackage{graphicx}
\usepackage[colorlinks=true, allcolors=blue]{hyperref}

\title{Extending the Gini Index to Higher Dimensions via Whitening Processes}

\author{
    Gennaro Auricchio\thanks{Department of Mathematics, University of Padua, Padua, Italy. Email: gennaro.auricchio@unipd.it} \and
    Paolo Giudici\thanks{Department of Economics and Management, University of Pavia, Pavia, Italy. Email: paolo.giudici@unipv.it} \and
    Giuseppe Toscani\thanks{Department of Mathematics, University of Pavia and Institute of Applied Mathematics and Information Technologies, Pavia, Italy. Email: giuseppe.toscani@unipv.it}
}

\date{}

\def\bx{{\textbf x}}
\def\bX{{\textbf X}}
\def\bY{{\textbf Y}}

\def\by{{\textbf y}}
\def\bm{{\textbf m}}

\def \bccc{{\textbf c}}

\def \erre{\mathbb{R}}

\newtheorem{theorem}{Theorem}

\newtheorem{corollary}{Corollary}

\newtheorem{example}{Example}
\newtheorem{definition}{Definition}

\begin{document}
\maketitle

%------
% Insert your abstract.
%------
\begin{abstract}
Measuring the degree of inequality  expressed by a multivariate statistical distribution  is a challenging problem, which appears in many fields of science and engineering.  In this paper, we propose to extend the well known univariate Gini coefficient to  multivariate distributions, by maintaining most of its properties.
Our extension is based on the application of whitening processes that possess the property of scale stability.

\end{abstract}
\maketitle

\vspace{1cm}


\section{Introduction}
The Lorenz curve \cite{Lor} and the Gini index \cite{Gini1,Gini2} are still the most important tools to measure the inequality (mutual variability) expressed by a statistical distribution, such as the distribution of income or wealth in a country \cite{Eli,Gho, Gho2, To1}.  However,  they are univariate instruments, so that, for a given random $n$-dimensional vector $\bX$ of scalar components $X_i$, $i = 1,2,\dots, n$, they are suitable to analyze the variables $X_i$ individually, ignoring the dependence structure they have as components of $\bX$. In reason of their importance in economical applications, there had been several efforts to extend the notions of Lorenz curve and Gini index to the multivariate case. The earliest approach, by means of methods of differential geometry,  is due to Taguchi \cite{Tagushi1,Tagushi2}. Further suggestions came by Arnold \cite{Arnold}, Arnold and Sarabia \cite{AS}, Gajdos and Weymark \cite{GW}, Grothe, K\"achele and Schmid \cite{Gro},  Koshevoy and Mosler \cite{KoshMos96,KoshMos97}, and Sarabia and Jorda \cite{Sarabia}. Unfortunately, as outlined in \cite{AS},  all these multivariate extensions are essentially determined by elegant mathematical considerations but often lack applicability and interpretability. 
In addition, these proposals do not possess some of the fundamental properties which are required to inequality measures, properties satisfied by the univariate Gini index. 

In a recent paper \cite{giudici2024measuring}, the possibility of measuring inequality of  multidimensional statistical distributions by resorting to Fourier transform \cite{Au1,Au2,To1}, has been investigated.
There, one of the key properties that a multivariate Gini index should possess has been identified in the \textit{scaling invariance} property on components \cite{hurley2009comparing}, which
is essential when trying to recover the value of the inequality index in a
multivariate phenomenon composed by different quantities, possibly measured in different
unit of measurement.

By resorting to the Mahalanobis distance \cite{Maha}, in place of the Euclidean distance, in \cite{giudici2024measuring}  a new version of the multivariate Gini coefficient satisfying the the scaling invariance property was proposed and studied.
Furthermore, it was shown that, for multivariate Gaussian distributions, the value of the proposed multivariate Gini index is related to the coefficient of variation introduced by Voinov and Nikulin \cite{VN}, as it does for the univariate case.
Owing to the fact that Mahalanobis distance is closely related to the process of \emph{whitening}  of a random vector, we will here extend the methods in \cite{giudici2024measuring}, leading 
to some generalizations of the Mahalanobis distance.
Among them, we will extract one which is particularly well suited to define a new multivariate Gini index, which appears  easy-to-handle and interpret.

Whitening is a linear transformation which, given a random $n$-dimensional vector $\bX =(X_1,\dots,X_n)^T$, of mean value ${\bm} = (m_1,\dots, m_n)^T$ and covariance matrix $\Sigma$, returns a new random vector $\bX^{*}$ whose entries are orthonormal, that is  the variance of each $X^{*}_i$ is $1$ and
the covariance of any $X^{*}_i$ and $X^*_j$ is null, whenever $i\neq j$. 
Considering that orthonormality among random variables greatly simplifies multivariate data analysis, both from a computational and a statistical standpoint, whitening is a critically important tool, most often employed in pre-processing. 
In essence, whitening is a generalization of standardisation,  a 
%procedure 
transformation
that is carried out by:
\begin{equation}
\label{sta}
{\bX}^{*} = V^{-1/2}{\bX}
\end{equation}
where the diagonal matrix $V = diag(var(X_1), var(X_2), \dots,var(X_n) )$ contains the variances of $X_i$, $i =1,2, \dots,n$. 
This results in a new random vector, namely $\bX^*$, whose components have unitary variance, that is $var(X^*_i ) = 1$, for every $i =1,2, \dots,n$.
Notice, however, that this transformation does not remove any correlation that the original entries of the vector $\bX$ possess. 
As we shall see, most whitening procedures lack the \emph{scale stable} property, which ensures that the whitened random vector remains unchanged if the components of the original random vector $\bX$ are scaled by a positive quantity.
This property is essential to obtain a multivariate inequality index that is scale invariant.

The content of this paper is as follows. In Section 2 we describe in details the whitening process.   Furthermore, we discuss the properties which are important in connection with a good definition of a multivariate inequality index. Then, in view of applications, in Section 3, we introduce the new multivariate Gini-type index, by discussing its main properties. Section 4 presents an application of the multivariate index to the study of market economic inequality.

\section{The Whitening Process}\label{sec:whitening}
\label{sec:SI}
In what follows, we denote with $\mathcal{P}(\erre^n)$ the set of $n$ dimensional random vectors.
Since every random vector is identified by its associated probability distribution, with a slight abuse of notation, we use the random vector $\bX$ and its associated probability measure $\mu$ interchangeably.
Moreover, we denote with $\mathcal{P}_{Id}(\erre^n)$ the subset of $\mathcal{P}(\erre^n)$ containing the random vectors whose covariance matrix is the identity matrix.
In its most generic form, a \textit{whitening process} is a map that, given a $n$-dimensional random  vector $\bX$, of positive mean ${\bm}$,  covariance matrix $\Sigma$, and characterized by a probability measure $\mu$, returns a new $n$-dimensional random vector whose covariance matrix is the identity, that is $\mathcal{S}:\mathcal{P}(\erre^n)\to\mathcal{P}_{Id}(\erre^n)$.
We say that a whitening process $\mathcal{S}$ is linear if, for any given $\bX\in\mathcal{P}(\erre^n)$, there exists a $n\times n$ square matrix that depends on $\bX$ through its distribution, namely $W_{\mu}$, such that
\begin{equation}
    \label{whi1}
    {\bX}^{*} = \mathcal{S}(\bX)= W_{\mu}{\bX}.
\end{equation}
The matrix $W_\mu$ is also known as whitening matrix (associated with $\mathcal{S}$) \cite{kessy2018optimal}.
If the covariance matrix of $\bX$, namely $\Sigma$, is invertible, then the whitening matrix in \eqref{whi1} must satisfy the identity $W_\mu\Sigma W_\mu^T = I$ and thus $W_\mu (\Sigma W_\mu^TW_\mu) = W_\mu$, which boils down to
\begin{equation}
    \label{whi2}
    W_\mu^TW_\mu =\Sigma^{-1}.
\end{equation}
We remark that, given a $n$-dimensional random vector $\bX$ whose covariance matrix is $\Sigma$, condition \eqref{whi2} does not determine uniquely the linear application that sends $\bX$ to a whitened vector. 
Indeed, the identity \eqref{whi2} does not fully identify $W_\mu$ but allows for rotational freedom.
For example, given a whitening matrix $W_\mu$, any $\widetilde W_\mu$ of the form
\[
\widetilde W_\mu = Z W_\mu,
\]
is a whitening matrix as long as $Z$ is an orthogonal matrix, \textit{i.e.} $Z^TZ = I$, since 
\[
(\widetilde W_\mu)^T\widetilde W_\mu=W_\mu^TZ^TZW_\mu=W_\mu^TW_\mu=\Sigma^{-1},
\]
hence $\widetilde W_\mu$ satisfies \eqref{whi2} regardless of the choice of $Z$. 
Consequentially, there are multiple ways to whiten a random vector, even if we restrict our attention to linear whitening processes \cite{li1998sphering}.

The \emph{Zero-phase Components Analysis} (ZCA) whitening transformation, also known as Mahalanobis whitening \cite{Maha}, is characterized by the matrix
\begin{equation}  
\label{ZCA}
W_\mu^{Maha} = \Sigma^{-1/2}.
\end{equation}
Note that, since $\Sigma$ is symmetric and positive definite, we can decompose it as  follows
\[
    \Sigma=Z\Theta Z^T,
\]
where $Z$ is the eigenmatrix associated with $\Sigma$, and $\Theta$ is the diagonal matrix whose diagonal contains the positive eigenvalues of $\Sigma$.
Since $\Theta$ is diagonal and its diagonal contains only positive values, we have that $\Sigma=(\Theta^{{1}/{2}}Z)^T(\Theta^{{1}/{2}}Z)$ and, therefore
\begin{equation}
\label{Ma}
    \Sigma^{-1/2}=Z^T\Theta^{-1/2} Z.
\end{equation}

\begin{example}
 Let ${\bX}$ be a $n$-dimensional Gaussian random vector, of mean ${\bm}$ and positive definite  $n\times n$ covariance matrix $\Sigma$, with associated probability density function
 \begin{equation*}
 % \label{Gauss}
f_\bX({\bx})  = \frac 1{(2\pi)^{n/2}( \det\Sigma^{-1})^{1/2}} \exp\left\{-  \frac 12 (\bx -\bm)^T\Sigma^{-1} (\bx-\bm) \right\}, 
\end{equation*}
Since the whitening process returns a new  $n$-dimensional Gaussian random vector ${\bX}^*$  of the same dimension $n$ and with unit diagonal \emph{white} covariance, the components of ${\bX}^*$ are uncorrelated. In the Gaussian case, this is equivalent to independence, since we can express the joint probability density function as a product of the marginals. Hence
$\bX^* $ is a $n$-dimensional Gaussian random  vector, of mean $\bm^* = {W_{\mu} \bm}$ and unit diagonal  covariance, with independent components. Indeed
\begin{align*}
% \label{Gauss-w}
f_{\bX^*} ({\bx}) &= \frac 1{(2\pi)^{n/2}} \exp\left\{-  \frac 12 (\bx -\bm^*)^T(\bx-\bm^*) \right\} = \prod_{i=1}^n \frac 1{(2\pi)^{1/2}} \exp\left\{-  \frac 12 (x_i -m_{i}^*)^2 \right\}.
\end{align*}
\end{example}

When it comes to measure the inequality of a probability distribution, not all whitening processes are, however, equal.
For example, only a few of the known whitening processes are \emph{Scale Stable}, i.e. such that the random vector $\bX^{*}$ obtained via the whitening process does not change if we multiply one or more entries of the pre-whitening vector $\bX$ by a positive constant.

\begin{definition}[Scale Stability]
    A whitening process $\mathcal{S}$ is said to be \emph{Scale Stable} if
    \[
        \mathcal{S}({\bX})=\mathcal{S}(Q{\bX}),
    \]
    for every random vector $\bX$ and for any diagonal matrix $Q=diag(q_1,q_2,\dots,q_n)$ whose diagonal elements are positive constants.
\end{definition}

Scale Stability is an essential property for any sparsity index whose definition relies on whitened data, as it is connected to the scale invariance of the index itself.
For this reason, we now show that  linear \emph{Scale Stable} whitening processes always exist and identify two of them.

\subsubsection*{The Choleski Whitening}
\textit{Cholesky whitening} is a process based on the Cholesky factorization of a definite positive matrix. 
In this case the whitening matrix is
\begin{equation}\label{Chol}
W_\mu^{Chol} =L^T, 
\end{equation}
where $L$ is the unique lower triangular matrix with positive diagonal values that satisfies \eqref{whi2}. 
Owing to the triangular structure of $W_\mu^{Chol}$, the whitening process it induces is Scale Stable, as the following result shows.

\begin{theorem}
\label{thm:CholSS}
    The Choleski whitening process is {Scale Stable}.
\end{theorem}

\begin{proof}
    Let $Q=diag(q_1,q_2,\dots,q_n)$ be a diagonal matrix such that $q_i>0$.
    Given a random vector $\bX$, let us denote with $\Sigma$ its covariance matrix; then, we have that the covariance matrix of $\bY=Q\bX$ is $Q\Sigma Q$.
    Let $L$ be the Choleski factorization of $\Sigma^{-1}$, i.e. $\Sigma^{-1}=L^TL$.
    It is easy to see that the inverse matrix of $Q\Sigma Q$ is $Q^{-1}\Sigma^{-1} Q^{-1}$. Hence we have
    \begin{equation}
        (Q\Sigma Q)^{-1}=(LQ^{-1})^T(LQ^{-1}).
    \end{equation}
    Since $L$ is lower triangular,  $LQ^{-1}$ is lower triangular as well, since the $i$-th row of $LQ^{-1}$ is equal to the $i$-th row of $L$ multiplied by $q_i^{-1}$. 
    Owing to the uniqueness of the Choleski factorization, we conclude that $LQ^{-1}$ is the Choleski factorization associated to $(Q\Sigma Q)^{-1}$.
    Finally, notice that
    \[
        LQ^{-1}(Q{\bX})=L{\bX},
    \]
    which concludes the proof.
\end{proof}

\subsubsection*{The correlation Whitening}
The correlation whitening, also known as \textit{Zero-Components Analysis} (ZCA-cor) employs a whitening matrix which derives from the correlation matrix.
In this case, given a random vector $\bX$, the whitening matrix is defined as:
\begin{equation}
\label{ZCA-cor}
W_\mu^{ZCA-cor} = P^{-1/2}V^{-1/2},
\end{equation}
where $P$ is the correlation matrix of $\bX$, and $V$ is the diagonal matrix introduced in \eqref{sta}.
Again, notice that $P^{-\frac{1}{2}}$ in \eqref{ZCA-cor} is not defined uniquely, thus there are multiple ZCA-cor matrices associated with the same $\bX\in\mathcal{P}(\erre^n)$.
Since the correlation matrix $P$ is scale invariant, it is easy to prove that the ZCA-cor whitening process is Scale Stable.

\begin{theorem}
\label{thm:ZCAcorrSS}
    The ZCA-cor whitening process is {Scale Stable}, regardless of how the square root of $P^{-1}$ is selected.
\end{theorem}

\begin{proof}
Without loss of generality, we show this result for a specific square root of $P^{-1}$ as our proof can be generalized to any square root of $P^{-1}$.
    Let $Q=diag(q_1,q_2,\dots,q_n)$ be a diagonal matrix such that $q_i>0$.
    Given a random vector $\bX$, let us denote with $\Sigma$ its covariance matrix; then, we have that the covariance matrix of $\bY=Q\bX$ is $Q\Sigma Q$.
    Let $P$ be the correlation matrix associated with $\bX$.
    It is easy to see that $P$ is definite positive and symmetric. 
    We then decompose $P$ as $P= O^T\Lambda O$, where $\Lambda$ is a diagonal ma\-trix containing the eigenvalues of $P$ and $O$ is the matrix containing the eigenvectors associated with $P$.
    Notice that $\Sigma=V^{\frac{1}{2}}P V^{\frac{1}{2}}$, where $V=diag(var(X_1), var(X_2), \dots, var(X_n))$.
    Moreover, the correlation matrix $P$ is scale invariant, thus the correlation matrix induced by $Q\bX$ is still $P$.
    Let us now define $\Lambda^{-\frac{1}{2}}O^TV^{-\frac{1}{2}}$ and consider ${\bX}_{*}=\Lambda^{-\frac{1}{2}}O^TV^{-\frac{1}{2}}\bX$. It is easy to see that the covariance matrix induced by ${\bX}_{*}$ is the identity matrix.
    Let us now consider ${\bY}=Q\bX$.
    The variance of each $Y_i$ is equal to $q_i^2$ times the variance of $X_i$, that is $var(Y_i)=q_i^2 var(X_i)$.
    Since the correlation matrix of $Y$ is the same as the correlation matrix of $X$, and since we have that the ZCA-cor whitening matrix induced by $\bY$ is $\Lambda^{-\frac{1}{2}}O^TV_{Q}^{-\frac{1}{2}}$, where $V_Q=diag(q_1^2var(X_1),q_2^2var(X_2),\dots,q_n^2var(X_n))=QVQ$, we infer that
    \[
        \Lambda^{-\frac{1}{2}}O^TV_{Q}^{-\frac{1}{2}}\bY=\Lambda^{-\frac{1}{2}}O^TV^{-\frac{1}{2}}Q^{-1}\bY =\Lambda^{-\frac{1}{2}}O^TV^{-\frac{1}{2}}\bX,
    \]
    which concludes the proof for the ZCA-cor  whitening.
\end{proof}

\subsubsection*{Counter example}
We remark that we not all the whitening processes are Scale Stable. 
Consider for example the Principal Components Analysis (PCA) whitening, a well known statistical pre-processing method, whose whitening matrix is defined as
\begin{equation}
    \label{PCA}
    W^{PCA} =\Theta^{-1/2}Z^T, 
\end{equation}
where $\Theta$ is the diagonal matrix containing the eigenvalues of the covariance matrix $\Sigma$, and $Z$ the corresponding (orthogonal) eigenvector matrix (e.g. \cite{friedman1987exploratory}).
The PCA transformation first rotates the variables using the eigenvector matrix of $\Sigma$.
This results in orthogonal components, but with different variances. 
To obtain whitened components, the rotated variables are then scaled by the square root of the eigenvalues via the matrix $\Theta^{-1/2}$. 
Note that, due to the sign ambiguity of the eigenvectors $Z$, the PCA whitening matrix given by \eqref{PCA} is not unique. 
However, adjusting the column signs in $Z$ such that elements on the diagonal of $\Theta$ are positive: all diagonal elements positive, results in a unique PCA whitening transformation with positive diagonal cross-covariance and cross-correlation. 
Notice that this procedure is different from the one defining the ZCA whitening process, since the ZCA first scales the entries, then rotates the variables, and then scales the variables according to the eigenvalues of the correlation matrix.
Despite its similarities with the ZCA, the PCA whitening is not Scale Stable, as the next example shows.\\
Let us consider a Gaussian random vector $\bX$, and let us set $\nu$ its probability measure. Its mean is $m_\bX=(1,1)$ and covariance matrix is
\[
\Sigma_\nu=\begin{pmatrix}
4 & -2 \\
-2 & 3
\end{pmatrix}.
\]
Since the eigenvalues of $\Sigma_\nu$ are $\theta_1=5.56$ and $\theta_2=1.44$ and their associated eigenvectors are $v_1=(0.78,0.61)$ and $v_2=(-0.61,0.78)$, respectively, from \eqref{PCA} we infer that
\[
    W^{PCA}_\nu=\begin{pmatrix}
\frac{1}{\sqrt{5.56}} & 0 \\
0 & \frac{1}{\sqrt{1.44}}
\end{pmatrix}\begin{pmatrix}
0.78 & -0.61 \\
0.61 & 0.78
\end{pmatrix}=\begin{pmatrix}
0.33 & -0.26 \\
0.51 & 0.66
\end{pmatrix}.
\]
Therefore, the random vector $W_\nu^{PCA}\bX$ is a Gaussian vector whose covariance is the identity matrix and its mean is $W_\nu^{PCA}m_\bX=(0.07,1.17)$.
Let us now consider the Gaussian random vector $\bY=(2X_1,X_2)$, that is the random vector obtained by multiplying the first entry of $\bX$ by two. Let us denote by $\tilde\nu$ its probability measure.
It is easy to see that $\bY$ is still a Gaussian random vector, whose mean is $m_\bY=(2,1)$ and whose covariance matrix is
\[
    \Sigma_{\tilde\nu}=\begin{pmatrix}
16 & -4 \\
-4 & 3
\end{pmatrix}.
\]
In this case, the eigenvalues of $\Sigma_{\tilde\nu}$ are $\eta_1=17.13$ and $\eta_2=1.87$ and their associated eigenvector are $w_1=(0.96,0.27)$ and $w_2=(-0.27,0.96)$, respectively.
In particular, the PCA whitening matrix associated with $Y$ is
\[
    W^{PCA}_{\tilde\nu}=\begin{pmatrix}
0.23 & -0.06 \\
0.20 & 0.71
\end{pmatrix}.
\]
Therefore, the random vector $W^{PCA}_{\tilde\nu} \bY$ follows a Gaussian distribution whose covariance matrix is the identity and whose mean is $W^{PCA}_{\tilde\nu} m_\bY=(0.40,1.11)$.
We then conclude that $W^{PCA}_\nu \bX\neq W^{PCA}_{\tilde\nu} \bY$ so that the PCA whitening is not Scale Stable.\\

\subsubsection*{The $p$-Mahanolobis Metrics}
Given a $n$-dimensional random vector $\bX$, we denote with ${\bm} = (m_1,m_2,\dots, m_n)^T$ its mean and with $\Sigma_\mu$ its positive definite $n\times n$ covariance matrix.
The Mahalanobis metric is then defined as
\begin{equation}\label{m_2}
    \bm_2({\bX})=\sqrt{{\bm}^T\Sigma^{-1}_\mu{\bm}}=\sqrt{(W_\mu{\bm})^T(W_\mu{\bm})}=||W_\mu{\bm}||_2,
\end{equation}
where $W_\mu$ is any whitening matrix.
Since for any whitening matrix $W_\mu$ we have that 
\[
    \bm^TW_\mu^TW_\mu\bm=\bm^T\Sigma^{-1}_\mu\bm
\]
the Euclidean norm of $W_\mu\bm$ does not depend on $W_\mu$.
This property, however, is lost if we consider other norms of the vector $W_\mu\bm$.
In this case, the choice of the whitening matrix $W_\mu$ affects the value of the norm, and it may not be Scale Stable.
To overcome this issue it suffices to consider a Scale Stable whitening.

\begin{definition}[The $l_p$ Mahalanobis norm]
\label{def:mahanolobislp}
    Let $\mathcal{S}$ be a \emph{Scale Stable} whitening procedure and let $\bX$ be a random vector.
    For any $p\ge 1$, we define the $l_p$-Mahalanobis norm induced by $\mathcal{S}$ of $\bX$ as follows
    \begin{equation}\label{M_p}
        M_p({\bX})=||\mathbb{E}[\mathcal{S}({\bX})]||_{p},
    \end{equation}
    where $\mathbb{E}[\mathcal{S}(\bX)]$ is the vector containing the mean of $\mathcal{S}({\bX})$.
\end{definition}
From a constructive viewpoint, note that, owing to Theorem \ref{thm:CholSS} and \ref{thm:ZCAcorrSS}, both Choleski and correlation whitening processes allow us to define a generalized Mahanolobis norm that is scale invariant, i.e.
\[
    M_p({\bX})=M_p(Q{\bX}),
\]
for every diagonal matrix $Q$ whose diagonal contains strictly positive values.

\section{A new multivariate Gini-type index }
\label{sec:Gini}

In this section, we show how to define a scaling invariant Gini index for multivariate distributions using the $l_p$-Mahalanobis norm introduced in Definition \ref{def:mahanolobislp}.
From now on we consider only the correlation whitening, as it allows to define a whitening process that, given a random non-negative vector $\bX$, returns a whitened vector that is also non-negative.
As we will see, the key to obtain such property is to select a suitable $P^{-\frac{1}{2}}$ in equation \eqref{ZCA-cor}.

\begin{definition}
For any $\bX$ random vector, let $W_\mu^{ZCA-cor}$ be the correlation whitening process associated with $\bX$.
Then, we define
\begin{equation}
\label{def:Ginipdef}
    G_p(\bX)=\frac{1}{2M_p(\bX)}\int_{\erre^n\times \erre^n}||W^{ZCA-cor}_{\mu}(\bx-\by)||_p\mu(d\bx)\mu(d\by),
\end{equation}
where, for every $p\ge 1$,$M_p(\bX)$ is the Mahalanobis metric (see Definition \ref{def:mahanolobislp}).
\end{definition}
Depending on the whitening matrix, i.e. on the function $\bX\to W_\mu^{ZCA-cor}$ that maps a random vector into its whitening matrix, the index $G_p$ satisfies the defining properties of an inequality index.
For example, if the components of $\bX$ are non-negative, the whitened vector $W_\mu^{ZCA-cor}\bX$ might not be non-negative, depending on how we select $P^{-\frac{1}{2}}$ in \eqref{ZCA-cor}.
We now characterize a way to associate any $\bX$ to a square root of $P^{-1}$ such that $W_\mu^{ZCA-cor}\bX$ is non-negative if and only if $\bX$ is non-negative.

\begin{theorem}
\label{thm:positivitypreserving}
    Given $\bX$ a random vector whose correlation matrix $P$ is invertible, let $P^{-1}$ denote the inverse of $P$.
    Moreover, let us decompose $P^{-1}$ as $P^{-1}=O\Lambda O^T$ where $O$ is the matrix containing the eigenvectors and $\Lambda$ is the diagonal matrix containing the eigenvalues of $P^{-1}$.
   Finally, let $V$ be the diagonal matrix containing the variances of $\bX$. 
   Then, the random vector
    \begin{equation}
        \label{eq:whitenincorrect}
        O\Lambda^{-\frac{1}{2}}O^TV^{-\frac{1}{2}}\bX
    \end{equation}
    is non-negative if and only if $\bX$ is non-negative.
\end{theorem}

\begin{proof}
    Let us consider $P^{-\frac{1}{2}}=O\Lambda^{-\frac{1}{2}}O^T$, where $O$ is the orthonormal matrix containing all the eigenvectors of $P^{-1}$ and let $\Lambda^{-1}=diag(\frac{1}{\lambda_1},\frac{1}{\lambda_2},\dots,\frac{1}{\lambda_n})$, where $\lambda_i$ are the eigenvalues of $P$.
If $\bX$ is a non-negative random vector then $O\Lambda^{-\frac{1}{2}}O^TV^{-\frac{1}{2}}\bX$ is also non-negative, since
\begin{enumerate}
    \item if $\bX$ is non-negative, then $V^{-\frac{1}{2}}\bX$ is non-negative;
    \item the rotation $O^T$ maps the set $\erre_+^n:=\{\bx \;\text{s.t.}\; x_i\ge 0\}$ to $O^T(\erre^n_+)$;
    \item the linear application induced by $\Lambda^{-\frac{1}{2}}$ maps $O^T(\erre^n_+)$ to itself;
    \item the rotation $O$ maps $O^T(\erre^n_+)$ back to $\erre_+^n$, which concludes the proof.
\end{enumerate}
\end{proof}
Leveraging on Theorem \ref{thm:positivitypreserving}, given a random vector $\bX$, from now on we will consider the correlation whitening process induced by $P^{-\frac{1}{2}}=O\Lambda^{-\frac{1}{2}}O^T$, and set
\begin{equation}
    \label{eq:Wcorr}
    W^{ZCA}_\mu=O\Lambda^{-\frac{1}{2}}O^TV^{-\frac{1}{2}}.
\end{equation}

%\begin{remark}
    Note that, if $\bX$ is a random vector whose covariance matrix is the identity, then $W^{ZCA}_\mu=Id$.    
%\end{remark}
We will then consider the family of multidimensional Gini indices induced by $W^{ZCA}_\mu$, so that
\[
    G_p(\bX)=\frac{1}{2M_p(\bX)}\int_{\erre^n\times \erre^n}||W^{ZCA}_\mu(\bx-\by)||_p\mu(d\bx)\mu(d\by),
\]
When $p=1$, the latter identity becomes particularly interesting as, in this case,  we can express $G_1$ as a convex combination of the one dimensional Gini indexes, which we denote with $G$, of the components of the vector $\bX^*=W^{ZCA}_\mu\bX$.

\begin{definition}[$l_1$ Gini Index, ZCA]
    Let $\bX$ be a random vector whose mean vector is ${\bm} = (m_1,m_2,\dots, m_n)^T$ and whose covariance matrix $\Sigma$ is positive definite.
    Moreover, we denote with $\mu$ the probability measure associated with $\bX$. 
    We define
    \[
        G_1(\bX)=\frac{1}{2\sum_{i=1}^n|(W_\mu^{ZCA}\bm)_i|}\int_{\erre^n\times \erre^n}\sum_{i=1}^n|(W_\mu^{ZCA}(\bx-\by))_i|\mu(d\bx)\mu(d\by).
    \]
\end{definition}

\begin{theorem}
\label{thm:decomposition}
    Let $\bX$ be a random vector of mean ${\bm} = (m_1,m_2,\dots, m_n)^T$ and positive definite  $n\times n$ covariance matrix $\Sigma$.
    Then, we have
    \begin{equation}
        \label{eq:ginidec}
        G_1(\bX)=\sum_{i=1}^n\frac{|m_i^*|}{\sum_{i=1}^n |m_i^*|}G\big((W_\mu^{ZCA}\bX)_i\big)
    \end{equation}
    where $m_i^*=(W_\mu^{ZCA}\bm)_i$ and $G(X_i^*)$ is the one dimensional Gini index of the $i$-th component of $\bX^*$.
    Furthermore, if the components of $\bX$ are non negative, we have that 
    \[
    0\le G_1(\bX)\le 1.
    \]
\end{theorem}

Equation \eqref{eq:ginidec} is the most important result of this paper, and establishes that the higher dimensional Gini index induced by the $l_1$ Mahalanobis norm is a convex combination of the $1$-dimensional Gini indexes of the random variable $\bX^*=W^{ZCA}_\mu\bX$.

\begin{proof}
    First, notice that the mean vector of $\bX^*$ is, by definition, $\bm^*:=W_\mu^{ZCA}\bm$, so that
    \[
        \sum_{i=1}^n |m_i^*|=\sum_{i=1}^n|(W^{ZCA}_\mu\bm)_i|.
    \]
    By definition of $G_1$, we have that
    \[
        G_1(X_i^*)=\frac{1}{2|m^*_i|}\int_{\erre^n\times \erre^n}|x_i^*-y_i^*|(N_i)_\#\mu(d\bx_*)(N_i)_\#\mu(d\by_*)
    \]
    where $\mu$ is the probability measure associated with $\bX$ and $N:\erre^n\to\erre^n$ defined as $N_i:\bx \to (W_\mu^{ZCA}\bx)_i$.
    By a change of variables, we have that
    \[
        G_1(X_i^*)=\frac{1}{2|m_i^*|}\int_{\erre^n\times \erre^n}|(W_\mu^{ZCA}(\bx-\by))_i|\mu(d\bx)\mu(d\by).
    \]
    By plugging the value of $G_1(X_i^*)$ in the right hand-side of \eqref{eq:ginidec}, we retrieve identity \eqref{eq:ginidec}.
    To conclude, notice that $G(X^*_i)\in[0,1]$ since $W_\mu^{ZCA}\bX$ is a non-negative random vector hence $G_1(\bX)\in[0,1]$ since $G_1(\bX)$ is a convex combination of values in $[0,1]$.
\end{proof}

It is insightful to interpret the result presented in Theorem \ref{thm:decomposition} in the context of what a whitening procedure does in to the multivariate data. 
When analysing a multidimensional statistical distribution, derived from a set of data, the measured quantities are typically interdependent. 
Consequently, the inequality expressed by the distribution, which measures how far apart are the individual multidimensional observations from each other,  cannot be assessed simply as a function of the one-dimensional Gini indexes, each applied to a different one-dimensional component of the distribution. 
However, by whitening the data, we can express the same inequality as a function of the standard one-dimensional Gini indexes, applied to the whitened one-dimensional components.

Indeed, Theorem \ref{thm:decomposition} indicates  a natural method to combine  the one-dimensional Gini indices by means of  a convex combination. The weights of the combination depend on the relative importance of the mean of the 
%principal 
whitened components.
Specifically, the weight assigned to the Gini index of the $i$-th 
%principal 
component of $\bX^*$ is proportional to its mean value, normalised by the sum of all mean values.
We formalize the properties of the $G_1$ inequality measure in the following Corollary.

\begin{corollary}
\label{crr:properties}
    Let $\bX$ be a random vector of mean ${\bm} = (m_1,m_2,\dots, m_n)^T$ and positive definite $n\times n$ covariance matrix $\Sigma_\mu$.
    Then following properties hold
    \begin{enumerate}
   \item \label{crr:pt1} For any $\epsilon>0$, there exists a random variable $\bX_\epsilon$ such that $G_1(\bX_\epsilon)\le \epsilon$.
        \item \label{crr:pt2} For any $\epsilon>0$, there exists a random variable $\bX_\epsilon$ such that $G_1(\bX_\epsilon)\ge 1 - \epsilon$.
        \item \label{crr:pt2.5} The $G_1$ is \textit{Scale Invariant}, that is $G_1(\mathcal{X})=G_1(\mathcal{X}_Q)$ for any diagonal matrix $Q=diag(q_1,q_2,\dots,q_n)$ with $q_j > 0$, and satisfies the \textit{Rising Tide} property, that is
        \[
            G_1(\bX+\bccc)\le G_1(\bX)
        \]
        for any positive vector $\bccc\in\erre^n_+$.

    \item \label{crr:pt4} Let us assume that $|m_1^*|$ is much larger than $\sum_{i=2}^n|m_i^*|$. 
    Then, 
    \[
        \lim_{|(m_1)_*|\to\infty}G_1(\bX) = G(X_1^*)
    \]
    that is $G_1(\bX)\sim G(X_1^*)$.
    This property tells us that if one of the uncorrelated component of $\bX^*$ dominates the others meanwise, then the overall inequality is mostly determined by its inequality.  
    \end{enumerate}
\end{corollary}

\begin{proof}
We divide the proof into four parts.
\textbf{Proof of Point (\ref{crr:pt1}).}
Let $\epsilon>0$ be fixed.
Consider $\bX=(X_1,X_2,\dots,X_n)$ be a random vector whose components are independent and identically distributed. Moreover, assume that each $X_i$ is distributed as follows
\[
    X_i=\begin{cases}
        0\quad\quad\quad &\textrm{with probability} \quad p=\frac{1}{2}\\
        2\quad\quad\quad &\textrm{with probability} \quad p=\frac{1}{2}.
    \end{cases}
\]
It is easy to check that $\mathbb{E}[X_i]=1$, $Var(X_i)=1$ for every $i=1,\dots,n$, and, by construction, $Cor(X_i,X_j)=0$ if $i\neq j$, therefore $\bX^*=W^{ZCA}_\mu\bX=\bX$.
Since each component $\bX^*$ are identically distributed, formula \eqref{eq:ginidec} boils down to
\[
    G_1(\bX)=G(X_1)=\frac{2}{2\mathbb{E}[X_i]}=1.
\]
Given $M>0$ let us define $\bX_M=(X_1+M,X_2+M,\dots,X_n+M)$.
By the same argument used above, we have that $\bX_M^*=\bX_M$ and that
\[
    G_1(\bX_M)=G(X_1+M)=\frac{1}{(1+M)},
\]
It is then easy to see that if $M>\frac{1}{2\epsilon}$, then $G_1(\bX_M)\le \epsilon$.
\textbf{Proof of Point (\ref{crr:pt2}).}
Let $\epsilon>0$ be fixed.
Consider $\bX=(X_1,X_2,\dots,X_n)$ be a random vector whose components are independent and identically distributed. Moreover, assume that each $X_i$ is distributed as follows
\[
    X_i=\begin{cases}
        0\quad\quad\quad &\textrm{with probability} \quad 1-p\\
        \frac{1}{\sqrt{p}(1-p)}\quad\quad\quad &\textrm{with probability} \quad p,
    \end{cases}
\]
where $p\in(0,1)$.
It is easy to see that, for every $p\in(0,1)$, the covariance matrix of $\bX$ is the identity matrix, thus $\bX^*=W^{ZCA}_\mu\bX=\bX$.
Moreover, we have that $\mathbb{E}[X_i]=\frac{\sqrt{p}}{1-p}$ and thus
\[
    G_1(\bX)=G(X_1)=\frac{1-p}{2\sqrt{p}}\frac{2p(1-p)}{\sqrt{p}(1-p)}=1-p.
\]
In particular, if $p\le \epsilon$, we have $G_1(\bX)\ge 1-\epsilon$.

\textbf{Proof of Point (\ref{crr:pt2.5}).}
The scale invariance follows directly from Theorem \ref{thm:ZCAcorrSS}.
Let us now consider the rising tide property.
Let $\bccc$ be a vector whose components are positive, that is $\bccc=(c_1,c_2,\dots,c_n)$, with $c_i\ge 0$.
Since $\bX$ and $\bX+\bccc$ have the same covariance matrix, it follows that $W^{ZCA}_\mu=W^{ZCA}_{\bX+\bccc}$ and that $\mathbb{E}[\bX+\bccc]=\mathbb{E}[\bX]+\bccc$.
In particular, we have that
\[
W^{ZCA}_{\bX+\bccc}(\mathbb{E}[\bX+\bccc])=W^{ZCA}_\mu(\mathbb{E}[\bX]+\bccc)=\bm^*+W^{ZCA}_\mu\bccc
\]
and thus
\begin{align*}
    G_1(&\bX+\bccc)\\
    &=\frac{1}{2\sum_{i=1}^n|(\bm^*+W^{ZCA}_\mu\bccc)_i|}\int_{\erre^n\times \erre^n}\sum_{i=1}^n|(W^{ZCA}_\mu(\bx+\bccc-(\by+\bccc)))_i|\mu(d\bx)\mu(d\by)\\
    &=\frac{1}{2\sum_{i=1}^n|(\bm^*+W^{ZCA}_\mu\bccc)_i|}\int_{\erre^n\times \erre^n}\sum_{i=1}^n|(W^{ZCA}_\mu(\bx-\by))_i|\mu(d\bx)\mu(d\by)\\
    &\le\frac{1}{2\sum_{i=1}^n|\bm^*_i|}\int_{\erre^n\times \erre^n}\sum_{i=1}^n|(W^{ZCA}_\mu(\bx-\by))_i|\mu(d\bx)\mu(d\by)= G_1(\bX),
\end{align*}
where the last inequality follows from the fact that $W^{ZCA}_\mu$ maps $\{\bx\in\erre^n\;\text{s.t.}\;x_i\ge 0\}$ into itself, thus $|(\bm^*+W^{ZCA}_\mu\bccc)_i|=|\bm^*_i|+|(W^{ZCA}_\mu\bccc)_i|\ge |\bm^*_i|$ for every $i=1,\dots,n$.

\textbf{Proof of Point (\ref{crr:pt4}).}
It follows from the following identity 
\[
\lim_{|(m_1)_*|\to\infty}\frac{|m_i^*|}{\sum_{j=1}^n |m_j^*|}=\begin{cases}
    1\quad\quad\text{if}\quad i = 1\\
    0\quad\quad\text{otherwise}
\end{cases}.
\]
\end{proof}

Finally, since any affine transformation of a Gaussian random vector is a Gaussian random vector, we can explicitly express $G_1$ as a function of the parameters of the Gaussian distribution. 

\begin{theorem}
\label{thm:gaussiangini}
   If $\bX$ is a Gaussian random vector whose mean is non-null, that is $\bm\neq(0,0)$, then also $\bX^*$ is a Gaussian random vector with non-null mean.
   Moreover, we have that
    \[
        G_1(\bX)=\frac{n}{\sqrt{\pi}\sum_{i=1}^n|(W^{ZCA}_\mu\bm)_i|}.
    \]
  
\end{theorem}

\begin{proof}
    It follows from the fact that any whitened multivariate Gaussian distribution is a Gaussian distribution with independent components and whose covariance matrix is the identity.
\end{proof}

This shows that the $G_1$ index of a Gaussian distribution is proportional to the inverse of the $\|\bm^*\|_1$.
However, the result in Corollary
\ref{crr:properties} does not hold, as Gaussian vectors take value on $\erre^n$. In this case, Theorem \ref{thm:gaussiangini} is to be interpreted as giving the explicit expression of a multidimensional coefficient of variation, which exhibits the same properties of that of Voinov and Nikulin \cite{Aer,VN}, which is obtained by evaluating the multivariate Gini index of a Gaussian random vector with respect to the 2-Mahalanobis Metric.

Finally, notice that Theorem \ref{thm:gaussiangini} can be generalized to any case in which the multivariate probability distribution allows an explicit computation of the one dimensional Gini index for its components.

\section{Application}\label{sec:Examples}
In this Section we show the practical importance of our proposed multivariate index of inequality, by means of the example introduced in \cite{giudici2024measuring}, which concerns the study of market economic inequality.\\
 A country market is unequal, from an economic viewpoint, when it is concentrated, that is,  when it presents a high inequality: few companies have a large size  and many have a  small size. 
\\ To measure market inequality, we need to specify how we can measure the size of a company, using publicly available data. For publicly listed companies we can consider, for example, the daily market capitalisation, the current number of employees and the yearly revenues. This information is publicly downloadable from the website companiesmarketcap.com, which contains, at the moment, the $8,081$ largest companies  in the world (by capitalisation). 
\\Table \ref{sumstat} reports the summary statistics for all companies, in terms of Market Capitalisation, Number of Employees and Revenues.
 
\begin{table}[htbp!]
  \centering
  \scriptsize
    \begin{tabular}{lcc}
    \hline 
 & \textbf{\textit{Mean}} & \textbf{\textit{Standard deviation}}\\
\hline
\textbf{\textit{MarketCap}} & $1.22\cdot10^{10}$ & $7.15 \cdot10^{10}$  \\
\textbf{\textit{Revenues}} & $6.86\cdot10^9$ & $2.49\cdot10^{10}$  \\
\textbf{\textit{Employees}} & $1,50 \cdot10^4$ & $5,26 \cdot10^4$ \\

\hline
    \end{tabular}
      \caption{Summary statistics}\label{sumstat}
\end{table}%

  Table \ref{sumstat} shows that, as expected, both the mean and standard deviation of Market Capitalisation and Revenues are much larger than those of the Number of Employees. In addition, the variability from the mean of the Market Capitalisation is about three times higher than that of the Revenues.
\\The three variables are not much correlated with each other, as their correlation matrix in 
Table \ref{corr} shows.

  \begin{table}[htbp!]
  \centering
  \scriptsize
    \begin{tabular}{llll}
    \hline 
 & \textbf{\textit{MarketCap}} & \textbf{\textit{Employees}} &\textbf{\textit{Revenues}}\\
\hline
\textbf{\textit{MarketCap}} & 1.000 & 0.010 & 0.102\\
\textbf{\textit{Employees}} & 0.010 & 1.000 & 0.036 \\
\textbf{\textit{Revenues}} & 0.102& 0.036 & 1.000 \\
\hline
    \end{tabular}
      \caption{Correlation matrix}\label{corr}
\end{table}%

 It is usually of interest to compare market inequality in different countries. 
 This can be done comparing the value of the Gini one dimensional indices.  For the sake of illustration, and without loss of generality, here we will measure market inequality at the overall level as well as for nine of the largest economies:  Canada, China, France, Germany, Italy, France, India, Japan, the United Kingdom and the United States. 
\\Table \ref{results} shows the calculation of the one dimensional Gini indexes, using either Market Capitalisation, Number of Employees and Revenues as the metrics with which to measure the size of the company.

 \begin{table}[htbp!]
  \centering
  \scriptsize
    \begin{tabular}{lccccc}
    \hline 
Countries &
       Number of companies & Gini MarketCap & 
Gini  Employees & Gini Revenues & $G_1$  \\
    \hline
    United States & 3652 & 0.886 & 0.845& 0.851 & 0.856 \\ 
    Canada & 395 & 0.794& 0.840 & 0.879 & 0.829 \\ 
        France & 119 & 0.767& 0.794 & 0.805& 0.789 \\
    Germany & 220 & 0.777 & 0.838 & 0.777 & 0.793 \\
        Italy & 86 & 0.638 & 0.783 & 0.829 & 0.737 \\
    United Kingdom & 258 & 0.754 & 0.813 & 0.828 & 0.794 \\
    China & 314 & 0.761 & 0.782 & 0.785& 0.776 \\
    India & 564 & 0.747 & 0.816 & 0.860 & 0.801 \\
    Japan & 350 & 0.667 & 0.771 & 0.714 & 0.715 \\
    \hline
    All & 8081  & 0.830 & 0.833 & 0.835 & 0.832\\
    \hline
    \end{tabular}
      \caption{Unidimensional Gini coefficients (referred to: MarketCap, Employee and Revenue) and multidimensional $G_1$ coefficient.}\label{results}
\end{table}%
From Table  \ref{results}, when all countries are considered, the three Gini one dimensional indices are very similar to each other.
\\Differently, when individual countries, are considered, there are remarkable differences. For example, if we consider market capitalisation, the United States is the most concentrated country, followed by Canada, Germany and France. Whereas if we measure size in terms of number of employees, the United States is followed by Canada, Germany and India. And, in terms of revenues, the most concentrated country appears Canada, followed by India, the United States and Italy.
\\It follows that we do not have a unique ranking of the countries, in terms of market inequality: it depends on how we define the size of a company: using market capitalisation, number of employees, or revenues. 
\\The intuition suggests that we should take all three scales into account, to attain a reliable ranking of the countries, in terms of market inequality. A multidimensional measure of inequality  is necessary.
\\The multidimensional $G_1$ index fills this gap. Table \ref{results} reports, in the foremost right column, the values of the $G_1$ index, obtained applying equation \eqref{eq:ginidec} to the whitening process defined as in \eqref{eq:Wcorr}.
\\The values of $G_1$ show that, considering all world countries, the multidimensional Gini index is equal to 0.82, in line with the individual values. 
\\More importantly, the multidimensional index gives a ranking of country inequality that take all three size measurements into account. The most unequal country (most concentrated market) is the United States, followed by Canada, in line with the results of the individual Gini indices for Market capitalisation and Employees. The third most concentrated market is India, owing to its high concentration of revenues. The United Kingdom, Germany and France follow, close to each other. The least unequal countries are Italy and Japan, characterised by many small and medium enterprises.
\\For completeness, we remark that the weights attributed to the individual indices, in Equation \eqref{eq:ginidec} are equal to $(0.335, 0.301, 0.363)$, respectively for Market capitalisation, Employees and Revenues. This means that the whitening process gives a slightly higher weight to the inequality in Revenues, followed by that in Market Capitalisation and, last, by that in Number of Employees.

\begin{paragraph}{Fundings.}
- This work has been carried out under the activities of the National Group of Mathematical Physics (GNFM). GT wish to acknowledge partial support by IMATI, Institute for Applied Mathematics and Information Technologies “Enrico Magenes”,  Pavia, Italy.
The work has also been partially supported by
 the European Union - NextGenerationEU, in the framework of the GRINS- Growing Resilient, INclusive and Sustainable (GRINS PE00000018).
The views and opinions expressed are solely those of the authors and do not necessarily reflect those of the European Union, nor can the European Union be held responsible for them.

\end{paragraph}

\bibliographystyle{alpha}
\bibliography{sample}

\end{document}